\documentclass[12pt,letterpaper,reqno]{amsart}

\usepackage[mathcal]{eucal}
\usepackage{mathrsfs}
\usepackage{amsmath}
\usepackage{amsfonts}
\usepackage{amssymb}
\usepackage[retainorgcmds]{IEEEtrantools}
\usepackage{enumitem}

\usepackage[margin=1in]{geometry}
\setlist{noitemsep,topsep=-\parskip}
\usepackage{microtype}

\usepackage{graphicx}
\usepackage[all,cmtip]{xy}
\usepackage{tikz}

\usepackage{hyperref}\hypersetup{colorlinks,citecolor=black,filecolor=black,linkcolor=black,urlcolor=black}
\numberwithin{equation}{section}
\usepackage[style=alphabetic,backend=bibtex]{biblatex}
\usepackage{amsthm}
\theoremstyle{definition}
\newtheorem{defn}{Definition}[section]
\newtheorem{cons}[defn]{Construction}
\theoremstyle{plain}
\newtheorem{thm}[defn]{Theorem}
\newtheorem*{thm*}{Theorem}
\newtheorem{lem}[defn]{Lemma}
\newtheorem{cor}[defn]{Corollary}
\theoremstyle{remark}
\newtheorem{rem}[defn]{Remark}
\newtheorem{ex}[defn]{Example}

\newcommand{\IE}[1]{\begin{IEEEeqnarray*}{#1}}
\newcommand{\IEE}{\end{IEEEeqnarray*}}

\newcommand{\angles}[1]{\left\langle #1\right\rangle}

\newcommand{\RR}{\mathbb{R}}
\newcommand{\ZZ}{\mathbb{Z}}
\newcommand{\QQ}{\mathbb{Q}}
\newcommand{\CC}{\mathbb{C}}

\newcommand{\PP}{\mathbb{P}}

\newcommand{\oh}{\mathcal{O}}
\renewcommand{\tilde}{\widetilde}

\DeclareMathOperator{\Nef}{Nef}

\DeclareMathOperator{\Pic}{Pic}

\DeclareMathOperator{\lcm}{lcm}
\DeclareMathOperator{\trunc}{trunc}

\title[Explicit constructions of short virtual resolutions of truncations]{Explicit constructions of\\ short virtual resolutions of truncations}
\author{Lauren Cranton Heller}
\address{Department of Mathematics, University of Nebraska, 203 Avery Hall, Lincoln, NE 68588}
\email{\href{mailto:lheller2@unl.edu}{lheller2@unl.edu}}
\urladdr{\url{https://lcrantonh.github.io/}}

\begin{document}

\begin{abstract}
	We propose a concept of truncation for arbitrary smooth projective toric varieties and construct explicit cellular resolutions for nef truncations of their total coordinate rings. We show that these resolutions agree with the short resolutions of Hanlon, Hicks, and Lazarev, which were motivated by symplectic geometry, and we use our definition to exhibit nontrivial homology in the commutative algebraic analogue of their construction.
\end{abstract}

\maketitle

Powers $\mathfrak m^\ell$ of the maximal homogeneous ideal $\mathfrak m$ in the coordinate ring of $\PP^n$ are central to the theory of syzygies on projective space.  They are resolved by Eagon--Northcott complexes, and these resolutions are linear \cite{BE75}.  Each $\mathfrak m^\ell$ is a truncation of the coordinate ring and defines the structure sheaf on $\PP^n$.

More generally, sheaves on a smooth toric variety $X$ correspond to (multi)graded modules over the total coordinate ring $S$ of $X$ \cite{Cox95}, and truncations have appeared in this context as well \cite{MS04,BCS22a}.  However, even in toric varieties with low dimension or Picard rank their resolutions are not well described, as no suitable analogue to the Eagon--Northcott complex has been constructed.  We fill this gap.

\vspace{-\parskip}

\begin{thm*}[\ref{thm:res}]
	For each $d\in\Nef X$ there is a complex $F_\bullet$ (see Construction~\ref{cons:F}) which
	\begin{itemize}
		\item  has a cellular structure supported on the polytope of monomials in $S_d$,
		\item  is a free resolution of the possibly larger monomial ideal $\trunc_d(S)$ (see Definition~\ref{def:trunc-S}),
		\item  has length at most $\dim X$,
		\item  agrees with an existing short virtual resolution from symplectic geometry (see below).
	\end{itemize}
\end{thm*}

On products of projective spaces the resolutions of truncated modules are eventually linear, and again appear in the construction of Tate resolutions \cite{EES15}.  A slightly weaker condition characterizes multigraded Castelnuovo--Mumford regularity \cite{BCS22a}.  Products of projective spaces have the relatively unusual property that all effective divisors are numerically effective (nef).  Otherwise the elements of sufficiently positive degree may not form a submodule, unlike the ideals $\mathfrak m^\ell$, an issue which is addressed by Definition~\ref{def:trunc-S}.

Products of projective spaces nevertheless suggest a connection between the free resolutions of truncations and the output of a certain Fourier--Mukai transform.  Fourier--Mukai transforms are a type of derived functor.  They were used by Berkesch, Erman, and Smith to exhibit \emph{short virtual resolutions}, meaning resolutions by line bundles with length at most the dimension of the product of projective spaces \cite{BES20}.  Work of Bayer, Popescu, and Sturmfels in \cite{BPS01} implies an extension of this proof to unimodular toric varieties $X$ (see Remark~\ref{rem:BPS}).

For more general toric $X$ the edges of the cell complex from \cite{BPS01} intersect to form additional vertices whose algebraic meaning has been mysterious.  Remarkably, the relevant interpretation has come from symplectic geometry, where analagous pictures arise from the Bondal--Thomsen stratification of the torus \cite{Bon06,Tho00}.  Our $\tilde F(d)$ agrees with a recent construction by \cite{HHL24} of short virtual resolutions on all smooth projective toric varieties.

\begin{thm*}[\ref{thm:FM-of-S}]
	For each $d$ the sheafification $\tilde F_\bullet(d)$ with $F_\bullet$ as above
	\begin{itemize}
		\item  is the output of the Fourier--Mukai transform of $\oh(d)$ with respect to the resolution of the diagonal of Hanlon, Hicks, and Lazarev,
		\item  contains only twists in the Bondal--Thomsen collection of $X$ (see Lemma~\ref{lem:terms}).
	\end{itemize}
\end{thm*}

It is not initially clear that the output of the Fourier--Mukai transform should be exact at the level of modules, so our result offers a novel source of free resolutions on toric varieties.

The resolution of the diagonal in \cite{HHL24} is one of several independent proofs of the existence of short virtual resolutions on all smooth projective toric varieties.  The resolution of the diagonal in \cite{And23} is also of the correct length, a related categorical fact is proven in \cite{FH23}, and a direct algebraic proof of the existence of short virtual resolutions can be found in \cite{BE24}.  From the perspective of commutative algebra the result generalizes Hilbert's Syzygy Theorem to the multigraded coordinate rings of smooth projective toric varieties, and was previously known for a number of specialized cases \cite{CK08,BES20,Yan21,BS24}.  Hanlon, Hicks, Lazarev, Favero, Huang, and Anderson were motivated by generation questions in the derived category which arise from homological mirror symmetry.

Given a general coherent sheaf $\mathcal F$ on a smooth projective toric variety $X$, constructing the known resolutions of length $\leq\dim X$ for $\mathcal F$ requires applying either a Fourier--Mukai transform (using an injective resolution) or the mirror functor for $X$, neither of which is easily computable.

In this paper we outline an explicit and straightforward procedure to construct the output of the Fourier--Mukai transform from Hanlon, Hicks, and Lazarev when applied to the structure sheaf $\oh(d)$ for $d$ numerically effective (nef).  See Section~\ref{sec:cellular} for the construction, Section~\ref{sec:example} for a fully worked example, and Theorem~\ref{thm:FM-of-S} for the isomorphism.

We propose that, as in the case of products of projective spaces, the monomial ideal generated by the labeled vertices in our construction should be considered a truncation of the total coordinate ring $S$ at $d$, as in Definitions \ref{def:trunc-S} and \ref{def:trunc-M}.  Notably, adding additional vertex labels besides the monomials of $S_d$ allows it to admit a clear cellular structure.  Our cellular resolutions are then the analogue of Eagon--Northcott complexes in this case.

Finally we use the output of the Fourier--Mukai transform for $S(d)$ to describe the homology of the short virtual resolutions of other modules, in Corollary~\ref{cor:homology}.  Based on existing constructions we would have expected exactness even at the level of modules.  However in Example~\ref{ex:homology} homology persists for all multiples of $d$.  This possibility supports the philosophy of \cite{BES20} that allowing irrelevant homology can produce new geometrically relevant resolutions.

\subsection*{Acknowledgments}

I would like to thank Jay Yang, who was heavily involved in early parts of this project.  Alexandra Seceleanu, Daniel Erman, and Greg Smith helped me along the way.  I also appreciate the friendly willingness of Andrew Hanlon, Jeff Hicks, Oleg Lazarev, and Reginald Anderson to introduce us to the field of symplectic geometry at an AIM workshop.  This work was partially supported by the NSF under grants DMS-1901848, DMS-2101225, and DMS-2401482, and by an AMS--Simons travel grant.  I was additionally supported by grant DMS-1928930 while in residence at the Mathematical Sciences Research Institute in the spring of 2024.

\section{Notation}

Let $X$ be a smooth projective toric variety with total coordinate ring $S$ and irrelevant ideal $B\subseteq S$.  Let $M$ be the lattice of characters of the torus and $N$ the lattice of one-parameter subgroups.  Use $\angles{m,u}$ to denote the pairing $M\times N\to\ZZ$.  We assume that $X$ is defined over an algebraically closed field by a fan of full dimension.

Let $u_1,\ldots,u_r\in N$ be the primitive rays in the fan for $X$.  Given $\alpha\in\ZZ^r$ write $x^\alpha$ for the monomial $\prod_{i=1}^r x_i^{\alpha_i}$ in $S$ where $x_i$ is the variable corresponding to the ray $u_i$.  Write $\alpha\geq 0$ when $\alpha_i\geq 0$ for all $i$.  Let $\lceil\alpha\rceil$ be the vector with components $\lceil\alpha_i\rceil$ and similarly for $\lfloor\alpha\rfloor$.

Denote by $\Nef X$ the set of numerically effective divisors in the Picard group $\Pic X$ of $X$, or equivalently the set of globally generated (base point free) line bundles \cite[Theorem~6.3.12]{CLS11}.  The global sections of a nef line bundle $\oh(d)$ correspond to the lattice points of a polytope in $M\otimes\RR$ \cite[p.~66]{Ful93}.

Given $d\in\Nef X$, a labeled cell complex $D$ and a free resolution $F_\bullet$ are defined in Section~\ref{sec:cellular} using an incidence function $\varepsilon$ on $D$ as \cite[\S 1]{BS98}.

Another labeled cell complex $E$ and locally free resolution of the diagonal $\tilde G_\bullet$ are defined in Section~\ref{sec:derived}, independent of $d$ but also using an incidence function $\varepsilon$.  We will use $\Phi$ for the corresponding Fourier--Mukai transform acting on modules, not sheaves (see Section~\ref{sec:derived}).  In particular we will be interested in $\Phi(S(d))$.

\section{Construction of cellular resolution}\label{sec:cellular}

Fix $d\in\Nef X\subset\Pic X$.  We will define a free resolution $F_\bullet$ by constructing a labeled cell complex $D$ with support the polytope of sections of $\oh(d)$ on $X$.

Tensor the presentation of the Picard group of $X$ with $\RR$ to obtain the following diagram, where $\ZZ^r\to\RR^r$ is given by inclusion and rows and columns are exact.  We also include the quotient of $M\otimes\RR$ by the image of $M$, which will be used in Section~\ref{sec:derived}.

\begin{align}\label{eq:diagram}
\xymatrix@=12pt{
	& 0\ar[d] & 0\ar[d]\\
	0\ar[r] & M\ar[r]\ar[d] & \ZZ^r\ar[r]\ar[d] & \Pic X\ar[r] & 0\\
	0\ar[r] & M\otimes\RR\ar[r]\ar[d] & \RR^r\\
	& (M\otimes\RR)/M\ar[d]\\
	& 0
	}
\end{align}

Fix a representative coset $\alpha+M$ for $d$ in $\ZZ^r$, i.e.\ so that $x^\alpha$ is a degree $d$ monomial in $S$.

Consider the (infinite) cell structure $C$ on $\RR^r$ whose codimension 1 skeleton is the union of all integral translates of the coordinate planes in $\ZZ^r$.  We can restrict the cells of $C$ to the linear space $\alpha\otimes 1+M\otimes\RR$ to obtain a cell structure on $\alpha\otimes 1+M\otimes\RR$.  We further intersect with the quadrant $\RR^r_{\geq 0}$ to obtain the cell complex $D$.

Since the boundary of $\RR^r_{\geq 0}$ is contained in the coordinate planes, $D$ is a subcomplex of the restriction of $C$.  Thus it has an induced cell structure.

Points in the intersection of $\alpha+M$ and $\ZZ^r_{\geq 0}$ are exactly the nonnegative vectors with the same image as $\alpha$ in $\Pic X$, i.e.\ the exponents of monomials of degree $d$ in $S$.  Thus $D$ is finite and contains these points among its vertices.  However additional vertices may arise in the induced cell structure, as cells in $C$ with dimension equal to the Picard rank of $X$ generally intersect $\alpha\otimes 1+M\otimes\RR$ in dimension 0.

\begin{rem}\label{rem:BPS}
	Bayer, Popescu, and Sturmfels show that lattice points are the only vertices exactly when $X$ is unimodular \cite[Proposition~2.2]{BPS01}.  Their construction in \cite[Theorem~6.2]{BPS01} then gives a short cellular resolution in terms of line bundles for a sufficiently positive twist of a coherent sheaf on $X$, and will agree with ours.
\end{rem}

Label each cell in $D$ with a monomial by choosing a vector $\alpha+\gamma$ in its interior and rounding up fractional exponents to $x^{\lceil\alpha+\gamma\rceil}=x^{\alpha+\lceil\gamma\rceil}$.  This is well defined because $\lceil\gamma\rceil\neq\lceil\gamma'\rceil$ would imply that $\gamma$ and $\gamma'$ are separated by a wall of $C$ in $\RR^r$.

\begin{cons}\label{cons:F}
Define the cellular complex $F_\bullet$ associated to the labeled cell complex $D$ as in \cite[\S 1]{BS98} using an incidence function $\varepsilon$.
\end{cons}

Note that $F_\bullet$ is finite because $D$ is finite.  We first prove some basic properties of $D$ and $F_\bullet$.  A fully computed example of the complex $F_\bullet$ can be found in Section~\ref{sec:example} and does not depend on Section~\ref{sec:derived}.

\begin{lem}\label{lem:lcms}
	If $\sigma\subseteq D$ is a cell whose boundary cells have labels $x^{\tau_1},\ldots,x^{\tau_s}$ then $\sigma$ has the label $\lcm\{x^{\tau_1},\ldots,x^{\tau_s}\}$, meaning the monomial with exponent the componentwise maximum of the vectors $\tau_1,\ldots,\tau_s$.
\end{lem}

\begin{proof}
	Fix $1\leq i\leq r$ and consider the $i$th coordinates $\tau_{1,i},\ldots,\tau_{s,i}$.  Let $n$ be their maximum and $n'$ their minimum.

	Note that the interior of $\sigma$ must be either disjoint from or contained within the hyperplane with $i$th coordinate equal to $n$.  If it is contained then its boundary is as well and $\sigma$ is labeled by $n=\tau_{1,i}=\cdots=\tau_{s,i}=n'$.

	If it is disjoint we must have $n-n'=1$, in order to avoid the hyperplane with $i$th coordinate equal to $n-1$.  Then $n>\gamma_i>n'=n-1$ for all $\gamma$ in the interior of $\sigma$, so $\lceil\gamma_i\rceil=n$.
\end{proof}

\begin{defn}\label{def:trunc-S}
	By the previous lemma all labels on $D$ are multiples of the labels on the vertices.  We call the ideal they generate the \emph{ceiling truncation} $\trunc_d(S)$ of $S$ at $d$.
\end{defn}

\begin{thm}\label{thm:res}
	The complex $F_\bullet$ is a cellular \emph{resolution} (of $\trunc_d(S)$).
\end{thm}

\begin{proof}
	By \cite[Proposition~1.2]{BS98} it suffices to show that for each $\beta\in\ZZ^r$ the subcomplex of $D$ induced by vertices with labels dividing $x^\beta$ is convex and thus contractible.

	Fix $\beta$ and let $B$ be the induced subcomplex.  Since $\beta$ is integral, $x^{\lceil\gamma\rceil}$ divides $x^\beta$ if and only if $\gamma_i\leq\beta_i$ for all $i$.  Thus $B$ consists of all points satisfying these inequalities, so it is the intersection of a convex subset of $\RR^r_{\geq 0}$ with the linear space $\alpha\otimes 1+M\otimes\RR$, so still convex.
\end{proof}

\section{Fourier--Mukai transform}\label{sec:derived}

Here we outline the construction of Hanlon, Hicks, and Lazarev in the language of cellular resolutions.  See \cite{HHL24} for the full generality, which applies to toric embeddings of smooth stacks and properly takes place in a lattice arising from the map.  For our purposes the authors resolve the diagonal sheaf on $X\times X$ using a lattice canonically isomorphic to $M$ and a stratification agreeing with the cell structure $C$ from Section~\ref{sec:cellular} (whose codimension 1 skeleton is the union in $\RR^r$ of all integral translates of the coordinate planes).

Specifically, restrict $C$ to the image of $M\otimes\RR$ in $\RR^r$ and consider the induced cell structure $E$ on the quotient $(M\otimes\RR)/M$, a real torus.  We will also need an incidence function $\varepsilon$ on $C$ which descends to the quotient in the sense that it is invariant under translation of pairs of open cells by $M$.

For $\gamma\in(0+M\otimes\RR)$ the label $x^{\lceil\gamma\rceil}$ is not invariant under translation by $M$, but its degree in $\Pic X$ is (see diagram \eqref{eq:diagram}).  If $\gamma$ is in the relative interior of a cell $\sigma$ and $\delta$ is in a boundary cell of $\sigma$ then $\frac{x^{\lceil\gamma\rceil}}{x^{\lceil\delta\rceil}}$ is also well defined up to translating both $\gamma$ and $\delta$ by $M$.

Let $R=S\otimes_\CC S$ be the total coordinate ring of $X\times X$.  Label the quotient in $E$ of the open cell containing $\gamma$ by $x^{\lceil\gamma\rceil}\otimes x^{\lceil-\gamma\rceil}$ and let $G_\bullet$ be the resulting cellular complex from $E$ using the incidence function $\varepsilon$, which has terms composed of summands:
\[R\left(-\deg\left[x^{\lceil\gamma\rceil}\otimes x^{\lceil-\gamma\rceil}\right]\right)=S(\deg x^{\lfloor-\gamma\rfloor})\otimes_\CC S(\deg x^{\lfloor\gamma\rfloor})\]

Note that in general it is not possible to choose a global set of representatives so that $\delta$ is in the boundary of the cell containing $\gamma$ in $C$ whenever that is true in $E$.  However it is possible to do so locally in order to define the map between the corresponding terms of $G_\bullet$:
\begin{align}\label{eq:G}
\xymatrix{
	S(\deg x^{\lfloor-\delta\rfloor})\otimes_\CC S(\deg x^{\lfloor\delta\rfloor}) &&&& S(\deg x^{\lfloor-\gamma\rfloor})\otimes_\CC S(\deg x^{\lfloor\gamma\rfloor})\ar[llll]_{\varepsilon(\gamma,\delta) x^{\lceil\gamma\rceil-\lceil\delta\rceil}\otimes x^{\lfloor\delta\rfloor-\lfloor\gamma\rfloor}}
}
\end{align}
It is also possible to have multiple maps between the same two cells if parts of the boundary are identified in the quotient---see Section~\ref{sec:example} for an example.

Hanlon, Hicks, and Lazarev show that $\tilde G_\bullet$ is a resolution of the diagonal, meaning a resolution of the structure sheaf of the diagonal subvariety in $X\times X$.

Short virtual resolutions for other modules are then constructed via Fourier--Mukai transforms, a type of derived functor that uses a resolution of the diagonal to produce resolutions of other complexes. We describe the process without motivation or proofs, which can be found for instance in \cite[\S 5.1 and \S 8.3]{Huy06}.  Denote the projections of $X\times X$ by $p$ and $q$:
\[\xymatrix@=12pt{
	& X\times X\ar[dl]_p\ar[dr]^q\\
	X && X
}\]
If $Q$ is a finitely generated $S$-module then the complex $Rp_*\left[\left(Lq^*\tilde Q\right)\otimes_L\tilde G_\bullet\right]$ is quasi-isomorphic to $\tilde Q$.  The terms of this complex can be computed using a spectral sequence for the derived pushforward and are direct sums of line bundles by the projection formula.

If $d$ is sufficiently positive then the spectral sequence for $\tilde Q(d)$ degenerates to a single row and the complex $Rp_*\left[\left(Lq^*\tilde Q(d)\right)\otimes_L\tilde G_\bullet\right]$ is a short resolution of $\tilde Q(d)$.  In particular the map \eqref{eq:G} becomes
\[\xymatrix{
	\oh(\deg x^{\lfloor-\delta\rfloor})\otimes_\CC\Gamma\left(X,\tilde Q(d+\deg x^{\lfloor\delta\rfloor})\right) &&& \oh(\deg x^{\lfloor-\gamma\rfloor})\otimes_\CC\Gamma\left(X,\tilde Q(d+\deg x^{\lfloor\gamma\rfloor})\right)\ar[lll]_{\pm x^{\lceil\gamma\rceil-\lceil\delta\rceil}\otimes x^{\lfloor\delta\rfloor-\lfloor\gamma\rfloor}}
}\]
on the second page.  From the perspective of commutative algebra we are interested in the complex of free $S$-modules associated to the Fourier--Mukai transform of $Q$, which can be obtained by applying the functor $\Gamma_*=\bigoplus_{d'\in\Pic X}\Gamma\left(X,-(d')\right)$ to this row.

\begin{cons}\label{cons:phi}
Denote $\Gamma_*$ of the Fourier--Mukai transform of $Q(d)$ with $\tilde G_\bullet$ by $\Phi(Q(d))$.
\end{cons}

This paper focuses on computing $\Phi(S(d))$, in which case $\tilde Q=\tilde S=\oh$ and $\Gamma\left(X,\oh(d')\right)$ is isomorphic to the degree $d'$ part of $S$, giving the module map
\begin{align}\label{eq:FM}
\xymatrix{
	S(\deg x^{\lfloor-\delta\rfloor})\otimes_\CC S_{d+\deg x^{\lfloor\delta\rfloor}} &&& S(\deg x^{\lfloor-\gamma\rfloor})\otimes_\CC S_{d+\deg x^{\lfloor\gamma\rfloor}}\ar[lll]_{\pm x^{\lceil\gamma\rceil-\lceil\delta\rceil}\otimes x^{\lfloor\delta\rfloor-\lfloor\gamma\rfloor}}
}
\end{align}
whose second coordinate is ordinary multiplication.  Assuming that $d$ is nef is sufficient for this to be the only row in the spectral sequence of the Fourier--Mukai transform by the following lemma.

We can always choose $\gamma\in\QQ^r$ so we will restrict ourselves to $\QQ$-divisors.  Use $\deg$ to represent both the standard degree map and the induced map on monomials with rational exponents, since $\Pic X$ has no torsion as $X$ is smooth.

\begin{lem}\label{lem:vanishing}
	If $d$ is nef and $e$ is of the form $\deg x^{\lfloor\gamma\rfloor}$ for some $\gamma$ in the image of $M\otimes\QQ$ then $H^i\left(X,\oh(d+e)\right)=0$ for $i>0$.
\end{lem}

\begin{proof}
	Choose $\alpha$ so that $\deg x^\alpha=d$, as in Section~\ref{sec:cellular}.
	Then $\deg x^{\alpha+\gamma}=d+0$ because $\gamma\in M\otimes\QQ$.  Thus by \cite[Theorem~9.3.5]{CLS11} we have $H^i\left(X,\oh(\deg x^{\lfloor\alpha+\gamma\rfloor})\right)=0$ for $i>0$, where $\deg x^{\lfloor\alpha+\gamma\rfloor}=\deg x^\alpha+\deg x^{\lfloor\gamma\rfloor}=d+e$ as desired.
\end{proof}

\section{Example}\label{sec:example}

Let $X$ be the Hirzebruch surface $\mathbb F_2$.  For simplicity we use the notation $S=\CC[x,y,z,w]$.  The fan for $X$ and the degrees of the variables in $S$ are illustrated in Figure~\ref{fig:fan}.

\begin{figure}[h]
  \begin{tikzpicture}[scale=.5]
    \path[use as bounding box] (-3,-3) rectangle (3,3);

    \fill[fill=blue!10] (0,0) -- (3, 0) -- ( 3,  3) -- ( 0, 3) -- cycle;
    \fill[fill=blue!25] (0,0) -- (0, 3) -- (-3/2,3) -- ( 0, 0);
    \fill[fill=blue!40] (0,0) -- (0,-3) -- (-3, -3) -- (-3, 3) -- (-3/2,3) -- (0,0);
    \fill[fill=blue!55] (0,0) -- (3, 0) -- ( 3, -3) -- ( 0,-3) -- cycle;

    \draw[line width=1pt,black] (0,0) -- (-3/2,3);
    \draw[line width=1pt,black] (0,0) -- ( 0, -3);
    \draw[line width=1pt,black] (0,0) -- ( 3,  0);
    \draw[line width=1pt,black] (0,0) -- ( 0,  3);

    \draw[line width=1.5pt,black,-stealth] (0,0) -- ( 1, 0) node[anchor=north west]{$\rho_0$};
    \draw[line width=1.5pt,black,-stealth] (0,0) -- ( 0,-1) node[anchor=south east]{$\rho_3$};
    \draw[line width=1.5pt,black,-stealth] (0,0) -- (-1, 2) node[anchor=north east]{$\rho_2$};
    \draw[line width=1.5pt,black,-stealth] (0,0) -- ( 0, 1) node[anchor=south west]{$\rho_1$};
  \end{tikzpicture}
  \hspace{1in}
  \begin{tikzpicture}[scale=.5]
    \path[use as bounding box] (-3,-2) rectangle (3,3);

    \fill[fill=blue!20] (0,0) -- (3,0) -- (3,3) -- (-3,3) -- (-3,3/2) -- cycle;
    \fill[fill=blue!35] (0,0) -- (3,0) -- (3,3) -- (0,3) -- cycle;

    \draw [thin, gray,] (-3,0) -- (3,0); 
    \draw [thin, gray,] (0,-1) -- (0,3); 

    \draw[line width=1.5pt,black,-stealth] (0,0) -- ( 1,0) node[anchor=south west]{$x,z$};
    \draw[line width=1.5pt,black,-stealth] (0,0) -- (-2,1) node[anchor=south west]{$y$};
    \draw[line width=1.5pt,black,-stealth] (0,0) -- ( 0,1) node[anchor=south west]{$w$};
  \end{tikzpicture}
  \caption{the fan of $X$ and the degrees of the variables of $S$}
  \label{fig:fan}
\end{figure}
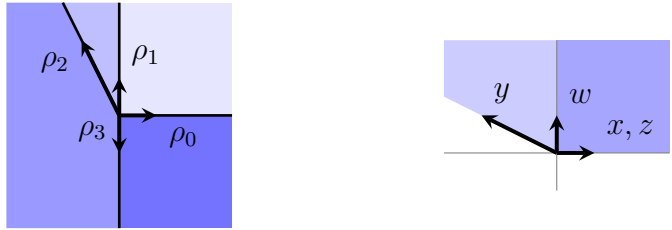

\noindent They form the presentation for $\Pic X\simeq\ZZ^2$ as follows:

\begin{align}\label{eq:SES}
\xymatrix{
	0\ar[r] & M\ar[rr]^{\begin{bmatrix}
		1 & 0\\
		0 & 1\\
		-1 & 2\\
		0 & -1
	\end{bmatrix}} && \ZZ^4\ar[rrr]^-{\begin{bmatrix}
		1 & -2 & 1 & 0\\
		0 & 1 & 0 & 1
	\end{bmatrix}} &&& \Pic X\ar[r] & 0
}
\end{align}

Fix $d=(1,1)$ and $\alpha=(0,0,1,1)$, corresponding to the monomial $zw$.  All monomials of degree $d$ fit into the trapezoid in Figure~\ref{fig:polytope}, defined by the hyperplanes where $\langle-,u_\rho\rangle=\alpha_\rho$.

\begin{figure}[h]
  \begin{tikzpicture}
    \fill[fill=blue!20] (0,0) -- (1,0) -- (3,1) -- (0,1) -- (0,0);
    
    \draw[line width=1pt,black] (3,1) -- (0,1) node[circle,fill,inner sep=2pt]{} node[above]{$yz^3$} -- (0,0) node[circle,fill,inner sep=2pt]{} node[below] {$zw$} -- (1,0) node[circle,fill,inner sep=2pt]{} node[below]{$xw$} -- (3,1) node[circle,fill,inner sep=2pt]{} node[above]{$x^3y$}; 
    \draw[black] (1,1) node[circle,fill,inner sep=2pt]{} node[above]{$xyz^2$};
    \draw[black] (2,1) node[circle,fill,inner sep=2pt]{} node[above]{$x^2yz$};
  \end{tikzpicture}
  \caption{the polytope of degree $d$ monomials in $S$}
  \label{fig:polytope}
\end{figure}
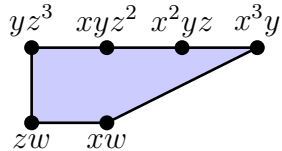

The restrictions of the coordinate planes in $\RR^4$ are then given by the hyperplanes where $\langle-,u_\rho\rangle$ is integer valued.  The labeled cell complex $D$ is pictured in Figure~\ref{fig:D}.  Recall that the plane containing $D$ is a coset of the image of $M\otimes\RR$ in $\RR^4$, with inclusion induced by the map $M\to\ZZ^4$ in the short exact sequence \eqref{eq:SES}.  Thus we can use $\alpha$ and the columns of the matrix to compute exponents for labels.  For example $\alpha+\frac{1}{2}(0,1,2,-1)=\left(0,\frac{1}{2},2,\frac{1}{2}\right)$ lies between $\alpha$ and $\alpha+(0,1,2,-1)$, and should be labeled $yz^2w$ from the ceiling $(0,1,2,1)$.

\begin{figure}[h]
  \begin{tikzpicture}
    \fill[fill=blue!20] (0,0) -- (2,0) -- (6,2) -- (0,2) -- (0,0);
    
    \draw[line width=1pt,black] (6,2) -- (0,2) node[above]{$yz^3$} -- (0,0) -- (2,0) -- (6,2) node[above]{$x^3y$};
    \draw[line width=1pt,black] (0,1) -- (2,2);
    \draw[line width=1pt,black] (0,0) node[below]{$zw$} -- (4,2);
    \draw[line width=1pt,black] (2,0) node[below]{$xw$} -- (2,2) node[above]{$xyz^2$};
    \draw[line width=1pt,black] (4,1) -- (4,2) node[above]{$x^2yz$};
	\draw[black] (0,1) node[left]{$yz^2w$};
	\draw[black] (2,1) node[below right]{$xyzw$};
	\draw[black] (4,1) node[below right]{$x^2yw$};
  \end{tikzpicture}
  \caption{the labeled CW complex $D$}
  \label{fig:D}
\end{figure}
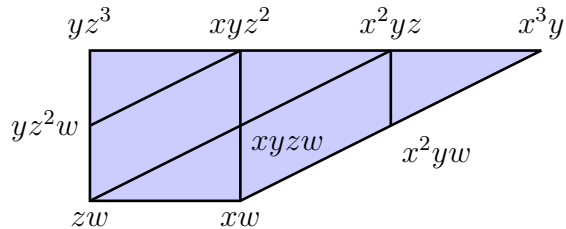

\begin{rem}
	In this example $S_d$ generates all of $\trunc_d(S)$, which is not true in general, yet unfortunately common to examples of a reasonable size for hand computation.
\end{rem}

We construct the complex $F_\bullet$ from $D$ (and an incidence function chosen by \emph{Macaulay2}).  By Lemma~\ref{lem:lcms} the remaining labels are the least common multiples of those on the vertices.  Refer to Figure~\ref{fig:table} for the correspondence between edges and summands in the middle column.

\[\xymatrix@R=0pt{
	S(-1,-1)^6 & S(-2,-1)^4 & S(0,-2)\\
	\oplus S(0,-2)^3 &\ar[l]_-{\partial_1}\oplus S(0,-2)^4 &\ar[l]_-{\partial_2}\oplus S(-1,-2)^2 &\ar[l] 0\\
	& \oplus S(-1,-2)^6 & \oplus S(-2,-2)^3\\
}\]
\[\partial_1 =
\scriptsize{\begin{bmatrix}
	-z &  0 &  0 &  0 &  0 &  0 &  0 &  0 &  w &  0 &  0 &  0 &  0 &  0\\
	x & -z &  0 &  0 &  0 &  0 &  0 &  0 &  0 &  w &  0 &  w &  0 &  0\\
	0 &  x & -z &  0 &  0 &  0 &  0 &  0 &  0 &  0 &  w &  0 &  w &  0\\
	0 &  0 &  0 & -z & xy &  0 & yz &  0 &  0 &  0 &  0 &  0 &  0 &  0\\
	0 &  0 &  x &  0 &  0 &  0 &  0 &  0 &  0 &  0 &  0 &  0 &  0 &  w\\
	0 &  0 &  0 &  x &  0 & xy &  0 & yz &  0 &  0 &  0 &  0 &  0 &  0\\
	0 &  0 &  0 &  0 & -1 &  0 &  0 &  0 & -x &  0 &  0 & -z &  0 &  0\\
	0 &  0 &  0 &  0 &  0 & -1 & -1 &  0 &  0 & -x &  0 &  0 & -z &  0\\
	0 &  0 &  0 &  0 &  0 &  0 &  0 & -1 &  0 &  0 & -x &  0 &  0 & -z
\end{bmatrix}},\
\partial_2 =
\scriptsize{\begin{bmatrix}
	0 &  0 &  0 & -w &  0 &  0\\
	0 &  0 &  0 &  0 & -w &  0\\
	0 &  0 &  0 &  0 &  0 & -w\\
	-y &  0 &  0 &  0 &  0 &  0\\
	0 & -z &  0 &  0 &  0 &  0\\
	1 &  0 & -z &  0 &  0 &  0\\
	-1 &  x &  0 &  0 &  0 &  0\\
	0 &  0 &  x &  0 &  0 &  0\\
	0 &  0 &  0 & -z &  0 &  0\\
	0 & -1 &  0 &  0 & -z &  0\\
	0 &  0 & -1 &  0 &  0 & -z\\
	0 &  1 &  0 &  x &  0 &  0\\
	0 &  0 &  1 &  0 &  x &  0\\
	0 &  0 &  0 &  0 &  0 &  x
\end{bmatrix}}\]

The cell complex $E$ is pictured in Figure~\ref{fig:E}, along with a portion of its labeled covering space in $M\otimes\RR$.  The underlying space of $E$ itself is a torus, where opposing edges of the fundamental domain are identified according to the arrows.  Again we compute labels from the map $M\to\ZZ^4$, this time changing signs in the second coordinate before taking ceilings.

\begin{figure}[h]
  \begin{tikzpicture}
	\fill[fill=blue!20] (0,0) -- (2,0) -- (2,2) -- (0,2) -- (0,0);
  	
  	\draw[->,line width=1pt,black] (0,0) -- (0,.5);
  	\draw[->,line width=1pt,black] (2,0) -- (2,.5);
  	\draw[->,line width=1pt,black] (0,1) -- (0,1.5);
  	\draw[->,line width=1pt,black] (2,1) -- (2,1.5);
  	\draw[->,line width=1pt,black] (0,0) -- (1,0);
  	\draw[->,line width=1pt,black] (0,2) -- (1,2);
  	
  	\draw[line width=1pt,black] (0,.5) -- (0,1);
  	\draw[line width=1pt,black] (2,.5) -- (2,1);
  	\draw[line width=1pt,black] (0,1.5) -- (0,2);
  	\draw[line width=1pt,black] (2,1.5) -- (2,2);
  	\draw[line width=1pt,black] (1,0) -- (2,0);
  	\draw[line width=1pt,black] (1,2) -- (2,2);
  	
  	\draw[line width=1pt,black] (0,1) -- (2,2);
  	\draw[line width=1pt,black] (0,0) -- (2,1);
  \end{tikzpicture}
  \begin{tikzpicture}  	
  	\fill[fill=blue!20] (0,0) -- (2,0) -- (2,2) -- (0,2) -- (0,0);
  	
  	\draw[line width=1pt,black] (2,2) -- (0,2) node[above]{$\frac{yz^2}{w}\otimes\frac{w}{yz^2}$} -- (0,0) -- (2,0) node[below]{$\frac{x}{z}\otimes\frac{z}{x}$} -- (2,2) node[above]{$\frac{xyz}{w}\otimes\frac{w}{xyz}$};
  	\draw[line width=1pt,black] (0,1) -- (2,2);
  	\draw[line width=1pt,black] (0,0) node[below]{$1\otimes 1$} -- (2,1);
  	\draw[black] (0,1) node[left]{$yz\otimes\frac{w}{z}$};
  	\draw[black] (2,1) node[right]{$x\otimes\frac{w}{x}$};
  \end{tikzpicture}
  \caption{the CW complex $E$ and its labeled covering space}
  \label{fig:E}
\end{figure}
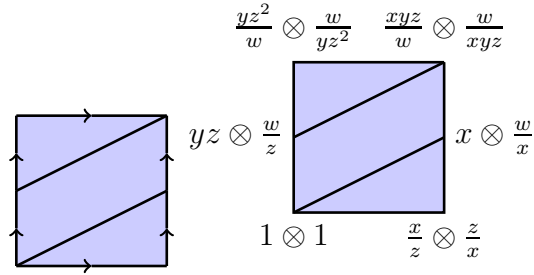

The terms of $G_\bullet$ correspond to the cells in $E$, with twist given by the label of any preimage.  To calculate each map we identify all adjacent preimages in the covering space, which may span multiple fundamental domains.

\[\xymatrix@R=0pt{
	R(0,0,0,0) & R(-1,0,-1,0) & R(1,-1,-1,-1)\\
	\oplus R(1,-1,1,-1) &\ar[l]_-{\partial_1} \oplus R(1,-1,0,-1)^2 &\ar[l]_-{\partial_2} \oplus R(0,-1,0,-1) &\ar[l] 0\\
	& \oplus R(0,-1,1,-1)^2 & \oplus R(-1,-1,1,-1)
}\]
\[\partial_1=\scriptsize{\begin{bmatrix}
		x\otimes z-z\otimes x & xy\otimes w & yz\otimes w & w\otimes xy & w\otimes yz\\
		0 & -1\otimes x & -1\otimes z & -x\otimes 1 & -z\otimes 1
\end{bmatrix}},\
\partial_2=\scriptsize{\begin{bmatrix}
		-y\otimes w & 0 & -w\otimes y\\
		1\otimes z & -z\otimes 1 & 0\\
		-1\otimes x & x\otimes 1 & 0\\
		0 & -1\otimes z & -z\otimes 1\\
		0 & 1\otimes x & x\otimes 1
\end{bmatrix}}\]

For instance, in the first entry of the first matrix above the horizontal edge in Figure~\ref{fig:E} maps to the same vertex in $E$ on both sides.  For one we divide the lcm label $x\otimes z$ by $1\otimes 1$ and for the other we divide $x\otimes z$ by $\frac{x}{z}\otimes\frac{z}{x}$.  We could also have used the equivalent labels from the top of Figure~\ref{fig:E}:
\[\frac{\frac{xyz^2}{w}\otimes\frac{w}{yz}}{\frac{yz^2}{w}\otimes\frac{w}{yz^2}}-\frac{\frac{xyz^2}{w}\otimes\frac{w}{yz}}{\frac{xyz}{w}\otimes\frac{w}{xyz}}=x\otimes z-z\otimes x\]
Our incidence function agrees with the one above when lifted to the covering space of $E$.

Finally we calculate the Fourier--Mukai transform $\Phi(S(d))$ using the formulas from Section~\ref{sec:derived}.  Recall that maps are given by the same matrices as in $G_\bullet$.

\[\xymatrix@R=0pt{
	 \left[S(0,0)\otimes S_{(1,1)}\right] & \left[S(-1,0)\otimes S_{(0,1)}\right] & \left[S(1,-1)\otimes S_{(0,0)}\right]\\
	\oplus\left[S(1,-1)\otimes S_{(2,0)}\right] &\ar[l] \oplus\left[S(1,-1)^2\otimes S_{(1,0)}\right] &\ar[l] \oplus\left[S(0,-1)\otimes S_{(1,0)}\right] \\
	& \oplus\left[S(0,-1)^2\otimes S_{(2,0)}\right] & \oplus\left[S(-1,-1)\otimes S_{(2,0)}\right]
}\]

We can see a correspondence of terms visually by comparing the placement of the cells of a given type in $D$ with the lattice points in the polytopes of degrees of $S$ appearing in $\Phi(S(d))$, as shown in Figure~\ref{fig:table} for the middle terms.  Lemma~\ref{lem:terms} will make this precise.

\begin{figure}[h]
  \begin{tabular}{|c|c|c|c|c|}\hline
  	edges in $D$ & $F_1$ & $G_1$ & $\Phi(S(d))_1$ & polytope in $S$\\\hline
	\begin{tikzpicture}[scale=0.5]
		\path[rectangle] (-.5,0) -- (6.5,2.5);
		\fill[fill=gray!20] (0,0) -- (2,0) -- (6,2) -- (0,2) -- (0,0);

		\draw[line width=1pt,gray] (0,1) -- (0,0) -- (2,0) -- (6,2) -- (0,2);
		\draw[line width=1pt,gray] (0,1) -- (2,2);
		\draw[line width=1pt,gray] (0,0) -- (4,2);
		\draw[line width=1pt,gray] (2,0) -- (2,1);

		\draw[line width=3pt,blue] (0,1) -- (0,2);
		\draw[line width=3pt,blue] (2,1) -- (2,2);
		\draw[line width=3pt,blue] (4,1) -- (4,2);
	\end{tikzpicture} & $S(-1,-2)^3$ & $R(0,-1,1,-1)$ & $S(0,-1)\otimes S_{(2,0)}$ &
	\begin{tikzpicture}
		\draw[line width=1pt,black] (0,0) node[circle,fill,inner sep=2pt]{} node[above] {$z^2$} -- (1,0) node[circle,fill,inner sep=2pt]{} node[above]{$xz$} -- (2,0) node[circle,fill,inner sep=2pt]{} node[above]{$x^2$};
	\end{tikzpicture}
	\\\hline
	\begin{tikzpicture}[scale=0.5]
		\path[rectangle] (-.5,0) -- (6.5,2.5);
		\fill[fill=gray!20] (0,0) -- (2,0) -- (6,2) -- (0,2) -- (0,0);

		\draw[line width=1pt,gray] (6,2) -- (0,2) -- (0,0) -- (2,0) -- (4,1);
		\draw[line width=1pt,gray] (0,0) -- (2,1);
		\draw[line width=1pt,gray] (2,0) -- (2,2);
		\draw[line width=1pt,gray] (4,1) -- (4,2);

		\draw[line width=3pt,blue] (0,1) -- (2,2);
		\draw[line width=3pt,blue] (2,1) -- (4,2);
		\draw[line width=3pt,blue] (4,1) -- (6,2);
	\end{tikzpicture} & $S(-1,-2)^3$ & $R(0,-1,1,-1)$ & $S(0,-1)\otimes S_{(2,0)}$ &
	\begin{tikzpicture}
		\draw[line width=1pt,black] (0,0) node[circle,fill,inner sep=2pt]{} node[above] {$z^2$} -- (1,0) node[circle,fill,inner sep=2pt]{} node[above]{$xz$} -- (2,0) node[circle,fill,inner sep=2pt]{} node[above]{$x^2$};
	\end{tikzpicture}
	\\\hline
	\begin{tikzpicture}[scale=0.5]
		\path[rectangle] (-.5,0) -- (6.5,2.5);
		\fill[fill=gray!20] (0,0) -- (2,0) -- (6,2) -- (0,2) -- (0,0);

		\draw[line width=1pt,gray] (2,0) -- (6,2);
		\draw[line width=1pt,gray] (0,0) -- (0,2);
		\draw[line width=1pt,gray] (0,1) -- (2,2);
		\draw[line width=1pt,gray] (0,0) -- (4,2);
		\draw[line width=1pt,gray] (2,0) -- (2,2);
		\draw[line width=1pt,gray] (4,1) -- (4,2);

		\draw[line width=3pt,blue] (0,0) -- (2,0);
		\draw[line width=3pt,blue] (0,2) -- (6,2);
	\end{tikzpicture} & $S(-2,-1)^4$ & $R(-1,0,-1,0)$ & $S(-1,0)\otimes S_{(0,1)}$ &
	\begin{tikzpicture}
		\fill[fill=blue!20] (0,0) -- (2,1) -- (0,1) -- (0,0);

		\draw[line width=1pt,black] (2,1) -- (0,1) node[circle,fill,inner sep=2pt]{} node[above]{$yz^2$} -- (0,0) node[circle,fill,inner sep=2pt]{} node[below] {$w$} -- (2,1) node[circle,fill,inner sep=2pt]{} node[above]{$x^2y$};
		\draw[black] (1,1) node[circle,fill,inner sep=2pt]{} node[above]{$xyz$};
	\end{tikzpicture}
	\\\hline
	\begin{tikzpicture}[scale=0.5]
		\path[rectangle] (-.5,0) -- (6.5,2.5);
		\fill[fill=gray!20] (0,0) -- (2,0) -- (6,2) -- (0,2) -- (0,0);

		\draw[line width=1pt,gray] (0,0) -- (2,0) -- (6,2) -- (0,2) -- (0,1);
		\draw[line width=1pt,gray] (0,1) -- (2,2);
		\draw[line width=1pt,gray] (0,0) -- (4,2);
		\draw[line width=1pt,gray] (2,1) -- (2,2);
		\draw[line width=1pt,gray] (4,1) -- (4,2);

		\draw[line width=3pt,blue] (0,0) -- (0,1);
		\draw[line width=3pt,blue] (2,0) -- (2,1);
	\end{tikzpicture} & $S(0,-2)^2$ & $R(1,-1,0,-1)$ & $S(1,-1)\otimes S_{(1,0)}$ &
	\begin{tikzpicture}
		\draw[line width=1pt,black] (0,0) node[circle,fill,inner sep=2pt]{} node[above] {$z$} -- (1,0) node[circle,fill,inner sep=2pt]{} node[above]{$x$};
	\end{tikzpicture}
	\\\hline
	\begin{tikzpicture}[scale=0.5]
		\path[rectangle] (-.5,0) -- (6.5,2.5);
		\fill[fill=gray!20] (0,0) -- (2,0) -- (6,2) -- (0,2) -- (0,0);

		\draw[line width=1pt,gray] (4,1) -- (6,2) -- (0,2) -- (0,0) -- (2,0);
		\draw[line width=1pt,gray] (0,1) -- (2,2);
		\draw[line width=1pt,gray] (2,1) -- (4,2);
		\draw[line width=1pt,gray] (2,0) -- (2,2);
		\draw[line width=1pt,gray] (4,1) -- (4,2);

		\draw[line width=3pt,blue] (0,0) -- (2,1);
		\draw[line width=3pt,blue] (2,0) -- (4,1);
	\end{tikzpicture} & $S(0,-2)^2$ & $R(1,-1,0,-1)$ & $S(1,-1)\otimes S_{(1,0)}$ &
	\begin{tikzpicture}
		\draw[line width=1pt,black] (0,0) node[circle,fill,inner sep=2pt]{} node[above] {$z$} -- (1,0) node[circle,fill,inner sep=2pt]{} node[above]{$x$};
	\end{tikzpicture}
	\\\hline
  \end{tabular}
  \caption{the correspondence of summands between $F_1$ and $G_1$}
  \label{fig:table}
\end{figure}

\section{Cellular description of transform for \texorpdfstring{$S(d)$}{S(d)}}

\begin{thm}\label{thm:FM-of-S}
	For $d$ nef the twist $F(d)_\bullet$ of the cellular resolution $F_\bullet$ on the labeled cell complex $D$ from Construction~\ref{cons:F} is isomorphic to the Fourier--Mukai transform $\Phi(S(d))$ of $S(d)$ with the resolution of the diagonal $G_\bullet$ of Hanlon, Hicks, and Lazarev, as described in Construction~\ref{cons:phi}, using incidence functions from the same $\varepsilon$ on $C$.
\end{thm}

We address the modules and the maps separately.  Note that the quotient of $D$ by the image of $M$ maps into $E\subseteq(M\otimes\RR)/M$ by identifying $M$ with the coset $\alpha+M$.  We will write points in $E$ as cosets $\gamma+M$ with $\gamma\in 0+M\otimes\RR\subseteq\RR^r$, so that the image of the point $\alpha+\gamma\in D\subseteq \alpha\otimes 1+M\otimes\RR$ in $E$ is $\gamma+M$.

\begin{lem}\label{lem:terms}
	Given $\gamma\in M\otimes\RR$ there is a one-to-one correspondence between the open cells in $D$ with image containing $\gamma+M\in E$ and the monomials of degree $d+\deg x^{\lfloor\gamma\rfloor}$ in $S$.
\end{lem}

\begin{proof}
	We construct a bijection.  Given an open cell $\sigma\subset D\subset\RR^r$ whose image contains $\gamma+M$ let $\beta\in\sigma$ be a preimage of $\gamma+M$ and consider the monomial $x^{\lfloor\beta\rfloor}$.  It is well defined because $\beta\geq 0$ by definition of $D$ and $\lfloor\beta\rfloor\neq\lfloor\beta'\rfloor$ would	imply $\beta$ and $\beta'$ not both in $\sigma$ (see, e.g., the proof of Lemma~\ref{lem:lcms}).

	Since the image $\beta-\alpha+M$ of $\beta$ in $E$ is equal to $\gamma+M$ we have $\beta-\alpha-\gamma\in M$.  Set $\delta=\beta-\alpha-\gamma$, so that
	\[\lfloor\beta\rfloor-\alpha=\lfloor\beta-\alpha\rfloor=\lfloor\gamma+\delta\rfloor=\lfloor\gamma\rfloor+\delta\]
	and thus
	\[\deg x^{\lfloor\beta\rfloor}=\deg x^\alpha+\deg x^{\lfloor\beta\rfloor-\alpha}-0=d+\deg x^{\lfloor\gamma\rfloor+\delta}-\deg x^\delta=d+\deg x^{\lfloor\gamma\rfloor}.\]

	Conversely, let $x^\xi$ be a monomial of degree $d+\deg x^{\lfloor\gamma\rfloor}$ in $S$.  Then $\deg x^\xi=\deg x^{\alpha+\lfloor\gamma\rfloor}$ so $\xi-\alpha-\lfloor\gamma\rfloor\in M$.  Define $\beta=\gamma+\xi-\lfloor\gamma\rfloor$, so that $\beta-\alpha+M=\gamma+M$.  We also have $\gamma-\lfloor\gamma\rfloor\geq 0$ and $\xi\geq 0$ so $\beta\geq 0$.  Finally $\beta-\alpha\in M\otimes\RR$ because $\gamma$ is in this set, so $\beta\in D$ is a preimage of $\gamma+M$ in some open cell.

	Furthermore
	\[\lfloor\beta\rfloor=\lfloor\gamma+\xi-\lfloor\gamma\rfloor\rfloor=\lfloor\gamma\rfloor+\lfloor\xi\rfloor-\lfloor\gamma\rfloor=\xi\]
	and if $\xi=\lfloor\beta\rfloor$ then
	\[\gamma+\xi-\lfloor\gamma\rfloor=\gamma+\lfloor\beta\rfloor-\lfloor\gamma\rfloor=\gamma+\alpha+\delta=\beta\]
	so these maps are inverses.
\end{proof}

Note that this correspondence does not match the labels we have chosen for the faces, which we obtained by rounding \emph{up}.  In general terms will be indexed by greatest common divisors but labeled by least common multiples.

We can think of the Fourier--Mukai transform as ``unfolding'' the cell complex on the torus across the polytope of monomials of degree $d$.  The dimensions of the vector spaces $S_{d+\deg x^{\lfloor\gamma\rfloor}}$ record the number of cells of each type occurring within the polytope, and multiplication in the second coordinate of the tensor product limits the map to adjacent cells in the cover.  For a visual representation of this see Figure~\ref{fig:table}.

\begin{proof}[Proof of Theorem~\ref{thm:FM-of-S}]
	Let $\sigma$ be a cell of $D$ and $\tau$ a boundary cell of $\sigma$.  Choose $\alpha+\gamma$ in the interior of $\sigma$ and $\alpha+\delta$ in the interior of $\tau$.

	In $F_\bullet$ these correspond directly to terms
	\[\xymatrix{
		S\left(-\deg x^{\alpha+\lceil\delta\rceil}\right) &&& S\left(-\deg x^{\alpha+\lceil\gamma\rceil}\right)\ar[lll]_{\varepsilon(\sigma,\tau)x^{\lceil\gamma\rceil-\lceil\delta\rceil}}
	}\]
	in homological index $\dim\sigma$ where $\dim\tau=\dim\sigma-1$.

	By Lemma~\ref{lem:terms} they also correspond to monomials $x^{\alpha+\lfloor\gamma\rfloor}$ and $x^{\alpha+\lfloor\delta\rfloor}$ of degrees $d+\deg x^{\lfloor\gamma\rfloor}$ and $d+\deg x^{\lfloor\delta\rfloor}$, respectively.  These give basis elements $1\otimes x^{\alpha+\lfloor\beta\rfloor}$ and $1\otimes x^{\alpha+\lfloor\delta\rfloor}$ in the domain and target of the map
	\[\xymatrix{
		S(\deg x^{\lfloor-\delta\rfloor})\otimes_\CC S_{d+\deg x^{\lfloor\delta\rfloor}} &&&& S(\deg x^{\lfloor-\gamma\rfloor})\otimes_\CC S_{d+\deg x^{\lfloor\gamma\rfloor}}\ar[llll]_-{\varepsilon(\sigma,\tau)x^{\lceil\gamma\rceil-\lceil\delta\rceil}\otimes x^{\lfloor\delta\rfloor-\lfloor\gamma\rfloor}}
	}\]
	from $\Phi(S(d))$ (see Section~\ref{sec:derived}), which occurs in the same homological index.  We can see that $F(d)_\bullet$ also has the right twists, as
	\[d-\deg x^{\alpha+\lceil\gamma\rceil}=d-d+\deg x^{-\lceil\gamma\rceil}=\deg x^{\lfloor-\gamma\rfloor}.\]
	The basis element $1\otimes x^{\alpha+\lfloor\gamma\rfloor}$ is sent to
	\[\varepsilon(\sigma,\tau)x^{\lceil\gamma\rceil-\lceil\delta\rceil}\otimes x^{\lfloor\delta\rfloor-\lfloor\gamma\rfloor}x^{\alpha+\lfloor\gamma\rfloor}=\varepsilon(\sigma,\tau)x^{\lceil\gamma\rceil-\lceil\delta\rceil}(1\otimes x^{\alpha+\lfloor\delta\rfloor}),\]
	so both maps are given by multiplication with $\varepsilon(\sigma,\tau)x^{\lceil\gamma\rceil-\lceil\delta\rceil}$ on the corresponding summands.  Since all summands and nonzero maps in both $F_\bullet$ and $G_\bullet$ arise in this way (see Lemma~\ref{lem:vanishing}) they are isomorphic by the assignment in Lemma~\ref{lem:terms}.
\end{proof}

\begin{cor}\label{cor:saturation}
	The $B$-saturation of $\trunc_d(S)$ is all of $S$.
\end{cor}

\begin{proof}
	The Fourier--Mukai transform of a complex of sheaves using a resolution of the diagonal is quasi-isomorphic to the original complex \cite[Example~5.4.(i)]{Huy06}.  Thus the sheafification of $\Phi(S(d))$ is quasi-isomorphic to $\oh(d)$, so its only homology is in index 0 and that homology is isomorphic to $\oh(d)$.  By Theorem~\ref{thm:FM-of-S} $\Phi(S(d))$ has homology $\trunc_d(S)(d)$, so $\widetilde{\trunc_d{S}(d)}\simeq\tilde S(d)$ which is equivalent to the desired statement by \cite[Theorem~3.7]{Cox95}.
\end{proof}

\section{Homology of short virtual resolutions}

We can use our definition of ceiling truncation for $S$ to truncate an arbitrary finitely generated $S$-module past a sufficiently positive degree depending on the module.  The resulting functor computes the homology of the Fourier--Mukai transform $\Phi$ for inputs other than twists of $S$.

\begin{defn}\label{def:trunc-M}
	Let $Q$ be a module with free presentation $K\gets L$ and let $d\in\Pic X$ be a degree with $d-a\in\Nef X$ for all summands $S(-a)$ of $K$ and $L$.  Define the \emph{ceiling truncation} $\trunc_d(Q)$ to be the cokernel of the map $\trunc_d(K)\gets\trunc_d(L)$, where each summand $S(-a)$ is truncated to $\trunc_d(S(-a))=\trunc_{d-a}(S)(-a)$ using Definition~\ref{def:trunc-S} and the map is induced by inclusion.
\end{defn}

\begin{rem}\label{rem:truncation}
	Other types of truncation exist in the literature, for instance in \cite[Definition~5.1]{MS04} and \cite[\S 4.4]{BC17}, which take the submodule of $Q$ generated by a subset of degrees.  In contrast, the ceiling truncation will not always be a submodule of $Q$.  In a sense we have forfeited this property for a better chance at exactness.  Indeed, if truncation is to be an exact functor it must satisfy the form of Definition~\ref{def:trunc-M}.
\end{rem}

We construct a truncation functor and see that it is still not exact.  However, the failure of exactness now corresponds to a previously unknown feature of the Fourier--Mukai transform.

\begin{lem}\label{lem:functor}
	For each $d\in\Pic X$, ceiling truncation $Q\mapsto\trunc_d(Q)$ is a functor on the category of modules $Q$ satisfying the hypotheses of Definition~\ref{def:trunc-M}.
\end{lem}

\begin{proof}
We first show that $\trunc_d(Q)$ is well-defined: that everything in $\trunc_d(L)$ maps to $\trunc_d(K)$ and that up to isomorphism $\trunc_d(Q)$ does not depend on the choice of presentation.

For the former it suffices to consider a component $S\gets S(-b)$ given by multiplication with a monomial $x^\beta$ of degree $b$.  According to Definition~\ref{def:trunc-S} a monomial in $\trunc_d(S(-b))=\trunc_{d-b}(S)(-b)$ has the form $x^{\lceil\delta\rceil}$ for $\delta\geq 0$ with $\deg x^\delta=d-b$ (using rational notation from Lemma~\ref{lem:vanishing}).  Then $\beta+\lceil\delta\rceil=\lceil\beta+\delta\rceil$, where $\beta+\delta$ is in the polytope of exponents of degree $d$ because $\deg x^{\beta+\delta}=d-b+b=d$ and $\beta+\delta\geq 0$.  Hence the image $x^\beta\cdot x^{\lceil\delta\rceil}$ is contained in $\trunc_d(S)$.

By \cite[Theorem~1.6]{Eis05}, each presentation $K'\gets L'$ of $Q$ will contain a minimal presentation as a summand, where all other summands have the form $0\gets S(-a)\gets S(-a)\gets 0$.  These trivial summands have trivial truncations, and the minimal presentation is unique up to isomorphism so the ceiling truncation is as well.

Finally a map $P\to Q$ of $S$-modules lifts to a map of presentations by \cite[Lemma~20.3]{Eis95} which induces a map $\trunc_d(P)\to\trunc_d(Q)$.  This respects composition because the maps on truncated presentations do, since they are simply restrictions of the original lifts.
\end{proof}

\begin{rem}
	Although we have only defined the ceiling truncation functor for sufficiently positive $d$, the author believes that a general definition is possible with a more technical construction.
\end{rem}

\begin{cor}\label{cor:homology}
	Let $Q$ be a finitely generated graded $S$-module and let $K_\bullet$ be a free resolution of $Q$.  If $d$ satisfies $d-a\in\Nef X$ for all summands $S(-a)$ in $K_\bullet$ then the transform $\Phi(Q(d))$ and the complex $\trunc_d(K_\bullet)(d)$ have isomorphic homology.
\end{cor}

\begin{proof}
	Without loss of generality replace $K_\bullet$ with $K_\bullet(d)$ and take $d=0$, since $\trunc_d(K_\bullet)(d)=\trunc_0\left((K_\bullet)(d)\right)$ by Definition~\ref{def:trunc-M}.  We will use the algebraic description of $\Phi(Q)$ from \eqref{eq:FM}.

	Let $\tilde G_\bullet$ be the HHL resolution of the diagonal as in Section~\ref{sec:derived} and consider the double complex $K'_\bullet\otimes G_\bullet$, where $K'_\bullet$ is the pullback of $K_\bullet$ to $R=S\otimes S$ by $q$.  Its columns are of the form $K'_j\otimes G_\bullet$.  Having twisted appropriately, we know from Theorem~\ref{thm:FM-of-S} that each $\Phi(K_j)$, the part of $K'_j\otimes G_\bullet$ in degree 0 for the second coordinate, has homology $\trunc_0(K_j)$ concentrated in the zeroth row.

	The rows of the double complex have the form $G_j\otimes K'_\bullet$, which has zeroth homology $G_j\otimes Q'$ for $Q'$ the pullback of $Q$, since $G_j$ is free.  In the spectral sequence of the double complex all maps are zero starting at the second page for both directions.  Focusing on degree 0 for the second coordinate, one direction gives the homology of the single row $\trunc_0(K_\bullet)$ and the other gives the homology of the single column $\Phi(Q)$.
\end{proof}

\begin{cor}
	The ceiling truncation of a resolution at $d$ has irrelevant homology for $d$ satisfying the hypothesis of Corollary~\ref{cor:homology}.
\end{cor}

\begin{proof}
	See the proof of Corollary~\ref{cor:saturation}.
\end{proof}

We can use this result to find a module $Q$ so that $\Phi(Q(d))$ has nontrivial homology even for $d$ an arbitrarily high multiple of an ample divisor.

\begin{ex}\label{ex:homology}
	Let $X$ be the toric variety with rays the rows of the matrix
	\[\begin{bmatrix}
		1 & 0 & 0\\
		0 & 1 & 0\\
		0 & 0 & 1\\
		-1 & 0 & -1\\
		1 & -1 & 1
	\end{bmatrix}\]
	and maximal cones spanned by the subsets $\{0,1,2\},\{0,1,3\},\{0,2,4\},\{0,3,4\},\{1,2,3\},\{2,3,4\}$. Let $d=(1,1)$ and consider $Q=S/I$ for $I=\angles{x_0x_1,x_1x_2}$.  The variety $X$ is smooth and $md$ satisfies the hypotheses of Corollary~\ref{cor:homology} for $m\geq 2$, yet $\Phi(Q(md))$ is not exact.

	This example was identified using the computer algebra system \emph{Macaulay2} \cite{M2}, but homology can be exhibited by hand for all $m$.  Its generators arise from the term $R(1,-1,1,-1)$ in $G_1$ and have the form $1\otimes x_1\cdot x^\beta$ where $\deg x^\beta =md+(-1,0)$.
\end{ex}

\printbibliography

\end{document}